\documentclass[11pt, a4paper]{article}
\usepackage{amsmath,amssymb,latexsym,graphics,epsfig}
\usepackage{hyperref}

\newtheorem{theo}{Theorem}[section]
\newtheorem{propo}[theo]{Proposition}
\newtheorem{lema}[theo]{Lemma}

\newtheorem{coro}[theo]{Corollary}

\input amssym.def
\newsymbol\rtimes 226F
\newfont{\nset}{msbm10}

\def\D{{\mbox {\boldmath $D$}}}

\def\F{{\cal F}}

\def\matrix0{{\mbox {\boldmath $O$}}}

\def\j{{\mbox{\boldmath $j$}}}
\def\m{{\mbox{\boldmath $m$}}}
\def\n{{\mbox{\boldmath $n$}}}

\def\r{{\mbox{\boldmath $r$}}}
\def\v{{\mbox{\boldmath $v$}}}
\def\w{{\mbox{\boldmath $w$}}}

\def\vec0{\mbox{\bf 0}}
\def\vecdelta{{\mbox{\boldmath $\delta$}}}

\def\vectheta{{\mbox{\boldmath $\theta$}}}

\def\dist{\mathop{d }\nolimits}
\def\dist{\mathop{\rm dist }\nolimits}

\begin{document}

\title{Moments in Graphs}

\author{C. Dalf\'{o}, M.A. Fiol, E. Garriga \\
         Departament de Matem\`{a}tica Aplicada IV \\
         Universitat Polit\`{e}cnica de Catalunya, BarcelonaTech\\
         \texttt{\{cdalfo,fiol,egarriga\}@ma4.upc.edu}}

\date{}

\maketitle

\begin{abstract}
Let $G$ be a connected graph with vertex set $V$ and a {\em weight function} $\rho$ that assigns a nonnegative number to each of its vertices. Then, the {\em $\rho$-moment} of $G$ at vertex $u$ is defined to be
$
M_G^{\rho}(u)=\sum_{v\in V} \rho(v)\dist (u,v)
$,
where $\dist(\cdot,\cdot)$ stands for the distance function. Adding up all these numbers, we obtain the {\em $\rho$-moment of $G$}:
$$
M_G^{\rho}=\sum_{u\in V}M_G^{\rho}(u)=\frac{1}{2}\sum_{u,v\in V}\dist(u,v)[\rho(u)+\rho(v)].
$$
This parameter generalizes, or it is closely related to, some well-known graph invariants, such as the {\em Wiener index} $W(G)$, when $\rho(u)=1/2$ for every $u\in V$, and the {\em degree distance} $D'(G)$, obtained when $\rho(u)=\delta(u)$, the degree of vertex $u$.

In this paper we derive some exact formulas for computing the $\rho$-moment of a graph obtained by a general operation called graft product, which can be seen as a generalization of the hierarchical product,  in terms of the corresponding $\rho$-moments of its factors. As a consequence, we provide a method for obtaining nonisomorphic graphs with the same $\rho$-moment for every $\rho$ (and hence with equal mean distance, Wiener index, degree distance, etc.). In the case when the factors are trees and/or cycles, techniques from linear algebra
allow us to give formulas for the degree distance of their product.
\end{abstract}


\section{Preliminaries}

In general {\em graphs invariants} based on the distance between vertices (in chemistry, called {\em topological indices}) have found
many applications in chemistry, since they give interesting correlations with physical, chemical and thermodynamic properties of molecules.
Some well-known examples are
the {\em Wiener index}  $W(G)$ (introduced by Wiener \cite{w47});
the first and second {\em Zagreb indices} $M_i(G)$ (Gutman and Trinajsti\'{c} \cite{gt80}, Zhou\cite{z04b,zg04}),
the {\em degree distance} (Dobrynin and Kotchetova \cite{dk94} and Gutman \cite{g94})
and the {\em molecular topological index} ${\rm MTI}(G)$ (proposed by Schultz \cite{sh89}).

Some results computing these indices for some graph operations (such as the Cartesian product, the join or the composition) and characterizing extremal cases have been given, among others, by
Bucicovschi and Cioab\u{a} \cite{bc08}, Eliasi and Taeri \cite{et09},
Khalifeh, Yousefi-Azari, Ashrafi, and Wagner \cite{kya08,kyaw09}, I. Tomescu \cite{t99}, A.I. Tomescu \cite{t08}, Yeh and Gutman \cite{yg94}, and Zhou \cite{z04a,z04b}. In particular, Stevanovi\'c \cite{s01} computed the so-called Wiener polynomial of a graph, from which the Wiener and hyper-Winer \cite{rt93b} indices are retrieved. Moreover, it is worth mentioning that some of these indices are closely related. For instance,
Klein, Mihali\'c, Plav\u{s}i\'c, and Trinajsti\'c \cite{kmpt92} proved that, when $G$ is a tree, there is a linear relation between ${\rm MTI}(G)$ and $W(G)$.
(See also Gutman \cite{g94,g02} for the study of other relations.)

As a generalization of most of the above indices, we define here the  $\rho$-moment of a graph by giving some weights to its vertices.
Then, we derive some exact formulas for computing the $\rho$-moment of a graph obtained by a general operation called `graft product', which can be seen as an extension of the hierarchical product \cite{bcdf09},  in terms of the corresponding $\rho$-moments of its factors. As a consequence, we provide a method for obtaining nonisomorphic graphs
that have the same $\rho$-moment for every $\rho$. In the case when the factors are trees and/or cycles, algebraic techniques
(distance matrices, eigenvalues, etc.) allow us to give formulas for the degree distance of their product. The remaining of this section is devoted to give some basic definitions and concepts on which our work relies.

\subsection{Graphs and moments}


Let $G$ be a (simple and finite) connected graph with vertex set $V=V(G)$, $n=|V|$ vertices and consider a {\em weight function}
$\rho : V\rightarrow [0,+\infty)$
that assigns a nonnegative number to each of its vertices. In particular, the {\em degree function} $\delta$ assigns to every vertex its degree.  The {\em $\rho$-moment}
of $G$ at a given vertex $u$ is defined as
$$
M_G^{\rho}(u)=\sum_{v\in V} \rho(v)\dist (v,u),
$$
where $\dist(\cdot,\cdot)$ stands for the {\em distance function}.
Adding up all these numbers, we obtain the {\em $\rho$-moment of $G$}:
\begin{eqnarray*}
M_G^{\rho} & = & \sum_{u\in V}M_G^{\rho}(u)
            =  \sum_{u\in V}\sum_{v\in V} \rho(v)\dist (v,u) \\
           & = &\sum_{v\in V} \rho(v) \sum_{u\in V} \dist (v,u)
            =  \frac{1}{2}\sum_{u,v\in V}\dist(u,v)[\rho(u)+\rho(v)].
\end{eqnarray*}

This parameter generalizes, or it is closely related to, some well-known graph invariants,
such as the following:
\begin{itemize}
  \item The {\em mean distance} $d(G)$ of $G$ is obtained when $\rho(u)=1$ for each $u\in V$:
$$
d(G)=\frac{1}{n^2}\sum_{u,v\in V}\dist(u,v)=\frac{1}{n^2}M_G^{1}.
$$
  \item The  {\em Wiener index} $W(G)$ \cite{w47} corresponds to the case $\rho(u)=1/2$ for every $u\in V$:
$$
W(G)=\frac{1}{2}\sum_{u,v\in V}\dist(u,v)=M_G^{1/2}.
$$
  \item The  {\em degree distance} $D'(G)$ proposed by Dobrynin and Kotchetova \cite{dk94} (see also Gutman\cite{g94} where it was denoted $S(G)$, I. Tomescu \cite{t99}, and A.I. Tomescu \cite{t08}) is obtained when $\rho(u)=\delta(u)$ for every $u\in V$, where $\delta$ stands for the {\em degree function}:
$$
D'(G) = \frac{1}{2}\sum_{u,v\in V}\dist(u,v)[\delta(u)+\delta(v)]=M_G^{\delta}.
$$
  \item The {\em Schultz index}, or {\em molecular topological index} MTI$(G)$ \cite{sh89}, is obtained by adding up the first {\em Zagreb index} $M_1(G)$ \cite{gt80}, which is the sum of the squares of the degrees and the degree distance:
$$
{\rm MTI}(G)=\sum_{u\in V}\delta(u)^2 + M_G^{\delta}.
$$
\end{itemize}

\subsection{The graft product}
As commented, our aim here is to obtain some exact formulas for computing the $\rho$-moment
of a graph, obtained by a `general' operation, which is defined as follows:
 Given the connected graphs $H$; $K_1,\ldots, K_r$  with
respective disjoint vertex sets $V_H$; $V_1,\ldots, V_r$
and some (root) vertices $x_i\in V_H$, $y_i\in V_i$, $i=1,\ldots,r$, the {\em graft product}
\begin{equation}\label{graft product}
G=H\left(
\begin{array}{ccc}
x_1 & \cdots & x_r \\
y_1 & \cdots & y_r
\end{array}
\right)(K_1,\ldots,K_r)
\end{equation}
is obtained by identifying vertices $x_i$ and $y_i$ for every $i=1,\ldots,r$, as it is represented in Figure \ref{fig1}.

\begin{figure}[t]
\begin{center}
\vskip-1cm
\includegraphics[width=12cm]{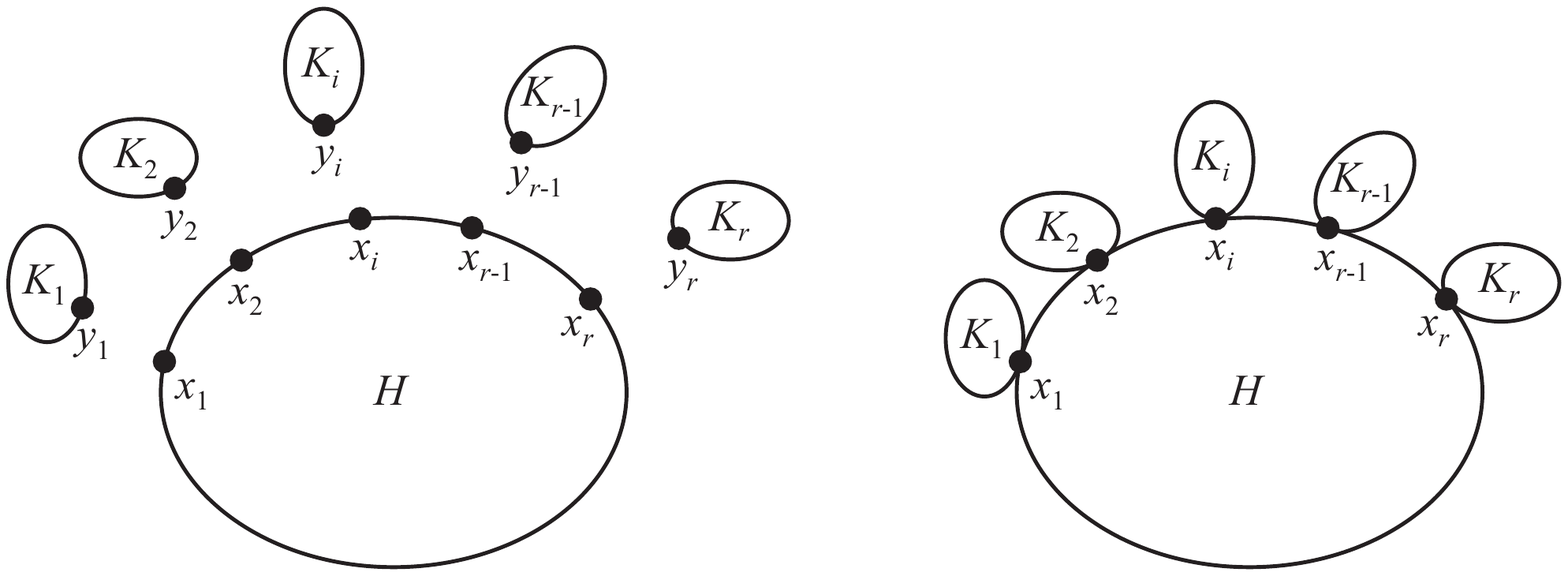}
\end{center}
\vskip -12.5cm
\caption{The graft product of graphs}
\label{fig1}
\end{figure}

Moreover, if $H;K_1,\ldots,K_r$ have weight functions $\alpha;\beta_1,\ldots,\beta_r$ respectively, we denote by $\gamma=\alpha+\beta_1+\cdots+\beta_r$ the weight function of their graft product $G$ defined as
$$
\gamma=\left\{
\begin{array}{ll}
\alpha(x) & \mbox{if } x\in H, \ x\neq x_i, \ 1\leq i\leq r,\\
\alpha(x_i)+\beta_i(x_i) & \mbox{for } 1\leq i\leq r,\\
\beta_i(y) & \mbox{if } y\in K_i, \ y\neq y_i, \ 1\leq i\leq r.
\end{array}
\right.
$$

In particular, when $r=1$, the so-called {\em coalescence} $H\!\cdot\! K$ of the `rooted graphs' $(H,x)$ and $(K,y)$ corresponds to the graft product
$H\!\cdot\! K=H{x\choose y}K$, which has been
studied in other contexts. For instance, Schwenk \cite{sch74} related the characteristic polynomial of $H\!\cdot\! K$ in terms of the characteristic polynomials of  $H$, $H-x$, $K$, and  $K-y$. Namely,
$$
\phi(H\!\cdot\!K)=\phi(H)\phi(K-y)+\phi(K)\phi(H-x)-x\phi(H-x)\phi(K-y).
$$
Then, by applying iteratively this formula, we can calculate the characteristic polynomial of a (general) graft product.

Another particular case in which the characteristic polynomial was studied is when $r=|V_H|$. In this case, we obtain the so-called
rooted product $H(K)$, where $K$ stands for the sequence $K_1,\ldots,K_r$ (for more details see Godsil and McKay \cite{GoMK78}).

\section{Main result}
In this section we derive our main result which gives a formula for computing the moment of a graft product in terms of the moments of its components.

\begin{theo}
\label{basic-theo}
Let $H$; $K_1,\ldots,K_r$ be graphs with respective (disjoint) vertex sets $V_H$; $V_1,\ldots,V_r$, weight functions
$\alpha; \beta_1,\ldots, \beta_r$  and (total) weights
$A=\sum_{u\in V_H}\alpha(u)$, $B_i=\sum_{v\in V_i}\beta_i(v)$, $i=1,\ldots, r$, and $B=\sum_{i=1}^r B_i$.
Then, the moment $M_G^{\gamma}$, where $G=(V,E)$ is the graft product (\ref{graft product}), with order $|V|=|V_H|+|V_1|+\cdots+|V_r|-r$, weight function $\gamma=\alpha+\beta_1+\cdots+\beta_r$ and weight $W=A+B$, is
\begin{eqnarray*}
M_G^{\gamma} & = & M_{H}^{\alpha}+\sum_{i=1}^r M_{K_i}^{\beta_i}
              +  \sum_{i=1}^r M_H^{\xi_i}(x_i)+ \sum_{i=1}^r M_{K_i}^{\eta_i}(y_i) \\
             & + & \sum_{i,j=1}^r(|V_i|-1)\dist(x_i,x_j)B_j,
\end{eqnarray*}
where
$\xi_i=(|V_i|-1)\alpha+B_i$
and
$\eta_i=(|V|-|V_i|)\beta_i+W-B_i$ for $i=1,\ldots,r$.
\end{theo}
\begin{proof}
We compute the moment $M_G^{\gamma}$ in three steps:

$(i)$ The moment of a vertex $v$ in $V_H$ is
$$
M_G^{\gamma}(v)
=\sum_{w\in V_H}\alpha(w)\dist(w,v)+\sum_{i=1}^r\sum_{w_i\in V_i}\beta_i(w_i)[\dist(w_i,x_i)+\dist(x_i,v)].
$$
Then,  by adding up for all $v\in V_H$, we get
\begin{eqnarray}
\sum_{v\in V_H}M_G^{\gamma}(v) & = & M_H^{\alpha}+|V_H|\sum_{i=1}^r M_{K_i}^{\beta_i}(x_i)+ \sum_{v\in V_H}\sum_{i=1}^r \dist(x_i,v)\sum_{w_i\in V_i} \beta_i(w_i)\nonumber\\
 & = & M_H^{\alpha}+|V_H|\sum_{i=1}^r M_{K_i}^{\beta_i}(x_i)+ \sum_{i=1}^r B_i\sum_{v\in V_H} \dist(x_i,v) \nonumber\\
 & = &
 \label{(i)} M_H^{\alpha}+|V_H|\sum_{i=1}^r M_{K_i}^{\beta_i}(x_i)+ \sum_{i=1}^r M_H^{B_i}(x_i).
\end{eqnarray}

$(ii)$ The moment of a vertex $v_i$ in $V_i$ is
\begin{eqnarray*}
M_G^{\gamma}(v_i) & = &
\sum_{w_i\in V_i}\beta_i(w_i)\dist(w_i,v_i)
+\sum_{w\in H}\alpha(w)[\dist(w,x_i)+\dist(x_i,v_i)] \\
& + & \sum_{j\neq i}\sum_{w_j\in V_j}\beta_j(w_j)[\dist(w_j,x_j)+\dist(x_j,x_i)+\dist(x_i,v_i)] \\
& = & \sum_{j=1}^r\sum_{w_j\in V_j}\beta_j(w_j)[\dist(w_j,x_j)+\dist(x_j,x_i)+\dist(x_i,v_i)] \\
& - & \sum_{w_i\in V_i} \beta_i(w_i)[\dist(w_i,x_i)+\dist(x_i,v_i)]+ \sum_{w_i\in V_i}\beta_i(w_i)\dist(w_i,v_i)\\
& + & \sum_{w\in V_H}\alpha(w)\dist(w,x_i)+\sum_{w\in V_H}\alpha(w)\dist(x_i,v_i)\\
& = & M_H^{\alpha}(x_i)+M_{K_i}^{\beta_i}(v_i)+\sum_{j=1}^r M_{K_j}^{\beta_j}(x_j)- M_{K_i}^{\beta_i}(x_i)\\
& + & (A+B-B_i)\dist(x_i,v_i)+ \sum_{j=1}^r B_j\dist(x_i,x_j).
 \end{eqnarray*}
Adding up first for all $v_i\in V_i$ and then for all $i=1,\ldots,r$, we get
\begin{eqnarray*}
\sum_{v_i\in V_i} M_G^{\gamma}(v_i) & = &
|V_i| M_H^{\alpha}(x_i)+M_{K_i}^{\beta_i}+|V_i|\sum_{j=1}^r M_{K_j}^{\beta_j}(x_j)- |V_i|M_{K_i}^{\beta_i}(x_i)\\
 & + & M_{K_i}^{A+B-B_i}(x_i)+ |V_i|\sum_{j=1}^r B_j\dist(x_i,x_j);
 \end{eqnarray*}
 \begin{eqnarray}
\sum_{i=1}^r\sum_{v_i\in V_i} M_G^{\gamma}(v_i) & = &
\sum_{i=1}^r |V_i| M_H^{\alpha}(x_i)+\sum_{i=1}^r M_{K_i}^{\beta_i}+\sum_{i=1}^r|V_i|\sum_{j=1}^r M_{K_j}^{\beta_j}(x_j)\nonumber\\
 & - & \sum_{i=1}^r |V_i|M_{K_i}^{\beta_i}(x_i)+
 \sum_{i=1}^r M_{K_i}^{A+B-B_i}(x_i)\nonumber\\
 &+& \label{(ii)}\sum_{i,j=1}^r |V_i| B_j\dist(x_i,x_j).
 \end{eqnarray}

$(iii)$ The vertices $x_i$ appear in both expressions (\ref{(i)}) and (\ref{(ii)}). Thus, we must compute their moments in order to subtract them from the total computation:
$$
M_G^{\gamma}(x_i) = \sum_{w\in V_H}\alpha(w)\dist(w,x_i)+\sum_{j=1}^r\sum_{w_j\in V_j}\beta_j(w_j)[\dist(w_j,x_j)+\dist(x_j,x_i)].
$$
Adding up for all $i=1,\ldots,r$,
\begin{eqnarray}
\sum_{i=1}^r M_G^{\gamma}(x_i) & = & \sum_{i=1}^r\sum_{w\in V_H}\alpha(w)\dist(w,x_i)+r\sum_{j=1}^r\sum_{w_j\in V_j}\beta_j(w_j)\dist(w_j,x_j)\nonumber \\
& + & \sum_{i,j=1}^r\dist(x_i,x_j)\sum_{w_j\in V_j}\beta_j(w_j)\nonumber \\
 & = & \label{(iii)}
 \sum_{i=1}^r M_H^{\alpha}(x_i)+r\sum_{j=1}^r M_{K_j}^{\beta_j}(x_j)+\sum_{i,j=1}^r B_j\dist(x_i,x_j).
\end{eqnarray}
Finally, from (\ref{(i)}), (\ref{(ii)}), and (\ref{(iii)}) we have
\begin{eqnarray*}
M_G^{\gamma} & = & \sum_{v\in V_H} M_G^{\gamma}(v)+\sum_{i=1}^r\sum_{v_i\in V_i} M_G^{\gamma}(v_i)-
\sum_{i=1}^r M_G^{\gamma}(x_i) \\
 & = & M_H^{\alpha}+\sum_{i=1}^r M_{K_i}^{\beta_i}+ \sum_{i=1}^r  |V_i| M_H^{\alpha}(x_i)
   -  \sum_{i=1}^r M_H^{\alpha}(x_i)+ \sum_{i=1}^r  M_H^{B_i}(x_i) \\
 & + & |V_H|\sum_{i=1}^r M_{K_i}^{\beta_i}(x_i)+
 \sum_{k=1}^r|V_k|\sum_{i=1}^r M_{K_i}^{\beta_i}(x_i)
    -   \sum_{i=1}^r |V_i| M_{K_i}^{\beta_i}(x_i)\\
 & - & r\sum_{i=1}^r M_{K_i}^{\beta_i}(x_i) +\sum_{i=1}^r M_{K_i}^{A+B-B_i}(x_i)+ \sum_{i,j=1}^r (|V_i|-1)B_j\dist(x_i,x_j)\\
 & = &
M_{H}^{\alpha}+\sum_{i=1}^r M_{K_i}^{\beta_i} +  \sum_{i=1}^r M_H^{(|V_i|-1)\alpha+B_i}(x_i)\\
 & + & \sum_{i=1}^r M_{K_i}^{(|V|-|V_i|)\beta_i+W-B_i}(y_i)
              +  \sum_{i,j=1}^r(|V_i|-1)B_j\dist(x_i,x_j),
\end{eqnarray*}
 which corresponds to our result.
\end{proof}

\section{Some consequences}
To discuss some consequences of the above result, let us consider some particular cases of the graft product.

\subsection{The flower graph}
The {\em  flower graph} is obtained when
 $H=x$ (a singleton). Then, $M_H^{\rho}=M_H^{\rho}(x)=0$ for any $\rho$ and Theorem \ref{basic-theo} gives
$$
M_G^{\gamma}  =  \sum_{i=1}^r M_{K_i}^{\beta_i}
              +  \sum_{i=1}^r M_{K_i}^{\eta_i}(y_i),
$$
where
$$
\eta_i=\left(\sum_{j\neq i}|V_j| -r+ 1\right)\beta_i+\alpha +\sum_{j\neq i}B_j,
$$
for $i=1,\ldots,r$.

\subsection{Graphs from permutations}
Another particular case is the family of so-called {\em graphs from permutations}, which are defined as follows:
Let $H=(V_H,\alpha,E_H)$ and $K_i=K=(V_K,\beta,E_K)$ for $i=1,\ldots,r$, with $V_H=\{x_1,\ldots,x_r\}$ and $V_K=\{y_1,\ldots,y_r\}$. Let $\sigma$ be a permutation of the indices $1,\ldots,r$ and consider the graph
$$
G_{\sigma}=H\left(
\begin{array}{ccc}
x_1 & \cdots & x_r \\
y_{\sigma(1)} & \cdots & y_{\sigma(r)}
\end{array}
\right)(K,\ldots,K).
$$
Then, in this case, Theorem \ref{basic-theo} yields:
\begin{coro}
\label{coro-sigma}
Let $G_{\sigma}$ be defined as above. Then,
$$
M_{G_{\sigma}}^{\gamma}  =  rM_H^{\alpha}+r^2M_K^{\beta}+r BM_{H}^{1}
              +  (A+(r-1)B) M_{K}^{1}.
$$
\end{coro}
\begin{proof}
With the notation of Theorem \ref{basic-theo}, we have
\begin{eqnarray*}
\xi_i(=\xi)=(|V_i|-1)\alpha + B_i & = & (r-1)\alpha + B, \\
\eta_i(=\eta) = W-B_i+(|V|-|V_i|)\beta_i & = & A + (r-1)B+(r^2-r)\beta.
\end{eqnarray*}
Hence,
\begin{eqnarray*}
&&\sum_{i=1}^r M_H^{\xi}(x_i)  = (r-1)\sum_{i=1}^r M_H^{\alpha}(x_i)+ B \sum_{i=1}^r M_H^{1}(x_i)
=(r-1)M_H^{\alpha}+ B M_H^{1}, \\
&&\sum_{i=1}^r M_{K}^{\eta}(y_{\sigma(i)}) = \sum_{i=1}^r M_{K}^{\eta}(y_{i}) = (A+(r-1)B)M_K^{1}+(r^2-r)M_K^{\beta},\\
&&\sum_{i,j=1}^r(|V_i|-1)\dist(x_i,x_j)B_j = (r-1)B\sum_{i,j=1}^r\dist(x_i,x_j)=(r-1)BM_H^1.
\end{eqnarray*}
Then,
\begin{eqnarray*}
M_{G_{\sigma}}^{\alpha+\beta+\cdots+\beta} & = & M_{H}^{\alpha}+r M_{K}^{\beta}+(r-1)M_H^{\alpha}+ B M_H^{1}+ (A+(r-1)B)M_K^{1}\\
 & + & (r^2-r)M_K^{\beta}+(r-1)BM_H^1\\
             & = & rM_H^{\alpha}+r^2M_K^{\beta}+r BM_{H}^{1}
              +  (A+(r-1)B) M_{K}^{1},
\end{eqnarray*}
as claimed.
\end{proof}

Consequently, we have that {\em the moment of $G_{\sigma}$ is independent of the permutation $\sigma$}.
This allows us to obtain nonisomorphic graphs with the same $\rho$-moment. Before giving an example of this fact, let us consider two interesting particular cases of Corollary \ref{coro-sigma}:
\begin{itemize}
\item[$(a)$]
If $\alpha=0$ and $\beta=1$ (constant), then $\gamma=1$ and Corollary \ref{coro-sigma} yields
\begin{equation}\label{permut-1}
 M^1_{G_{\sigma}}=r^2M_K^1+r^2 M_H^1+(r-1)rM_K^1=r^2 M_H^1+r(2r-1)M_K^1.
\end{equation}
Consequently, we get that the mean distance of $G_{\sigma}$ is
$$
d(G_{\sigma})=d(H)+\left(2-\frac{1}{r}\right)d(K)\quad \stackrel{r\rightarrow \infty}{\longrightarrow} \quad d(H)+2d(K).
$$
\item[$(b)$]
If $\alpha=\beta=\delta$ (the degree function), then also $\gamma=\delta$, and Corollary \ref{coro-sigma} gives that the degree distance of $G_{\sigma}$ is
\begin{equation}\label{permut-d}
 M^{\delta}_{G_{\sigma}}=rM_{H}^{\delta}+r^2M_K^{\delta}+2rm_K M_H^1+2(m_H+(r-1)m_K)M_K^1,
\end{equation}
where $m_H$ and $m_K$ stand for the size (number of edges) of $H$ and $K$, respectively.
\end{itemize}

Now, to give an example of non isomorphic graphs with the same $\rho$-moment, let us consider the graphs $H,K$ shown in  Figure \ref{fig2-moments}, with moments:
\begin{itemize}
\item
$M_H^{1}=2(1+1+2)+2(1+1+1)=14$,\quad  $M_H^{\delta}=2(3+3+4)+2(2+2+3)=34$,
\item
$M_K^{1}=2(1+2+3)+2(1+1+2)=20$,\quad $M_K^{\delta}=2(2+4+3)+2(1+2+2)=28$.
\end{itemize}

Then, we can choose three permutations $\sigma_i$  leading to the nonisomorphic graphs $G_{\sigma_i}$, $i=1,2,3$, shown in Figure~\ref{fig2-moments}, whose common moments with respect to $\gamma=\alpha+\beta+\beta+\beta+\beta$ turn out to be:
\begin{itemize}
\item
If $\alpha=0$ and $\beta=1$,
$$
M_{G_{\sigma_i}}^{1}=16\cdot 14+28\cdot 20 = 784, \quad\mbox{and}\quad d(G_{\sigma_i})=49/16, \quad i=1,2,3.
$$
\item
If $\alpha=\delta$ and $\beta=1$,
$$
M_{G_{\sigma_i}}^{\gamma}=4\cdot 34+16\cdot 20+4\cdot 4\cdot 14+(10+3\cdot 4)20 = 1120, \quad i=1,2,3.
$$
\item
If $\alpha=\beta=\delta$,
$$
M_{G_{\sigma_i}}^{\delta}=4\cdot 34+16\cdot 28+4\cdot 6\cdot 14+(10+3\cdot 6)20 = 1520, \quad i=1,2,3.
$$
\end{itemize}

Thus, in particular, the three graphs $G_{\sigma_i}$ have a common mean distance, Wiener index (since $W(G)^{1/2}=\frac{1}{2}M_G^{1}$), and degree distance.

\begin{figure}[t]
\begin{center}
\vskip-1cm
\includegraphics[width=12cm]{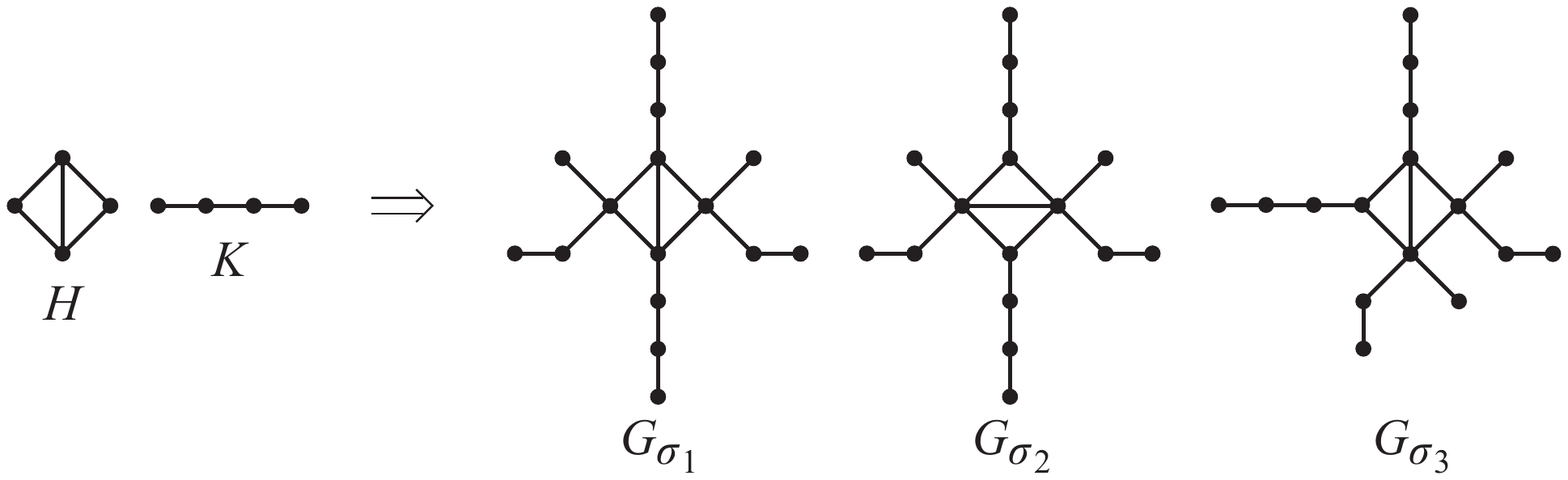}
\end{center}
\vskip -13cm
\caption{Three nonisomorphic graphs with the same $\rho$-moment}
\label{fig2-moments}
\end{figure}

\subsection{The partial hierarchical product}
Another family of interesting graphs are those obtained through the {\em partial hierarchical product} which is defined as follows:
Given the graphs $H$ and ($r$ copies of) $K$, and the vertices $x_1,\ldots,x_r\in V_H$, $y\in V_K$, consider the graph
$$
G_1=H\left(
\begin{array}{ccc}
x_1 & \cdots & x_r \\
y & \cdots & y
\end{array}
\right)(K,\ldots,K),
$$
see $G_1$ in Figure \ref{fig3-moments}.

In particular, when $r=|V_H|$, $G_1$ turns out to be the {\em hierarchical product $G_1=H \sqcap K$}, introduced by Barri\`ere, Comellas, Dalf\'o, and Fiol in \cite{bcdf09}, with vertex set $V_H\times V_K$ and adjacencies
$$
x_i y_j \sim
\left\{
\begin{array}{ll}
x_i y_k & \mbox{if $y_k\sim y_j$ in $K$}, \\
x_k y_j & \mbox{if $x_k\sim x_j$ in $H$ and $y_j=y$}.
\end{array}
\right.
$$

This is a spanning subgraph of the well-known direct (or Cartesian) product $H \square K$.
Moreover, $K_2\sqcap K_2\sqcap \stackrel{(n)}{\cdots}\sqcap K_2$ is the so-called {\em $n$-th binomial tree}, which is well-known in computer science as a model for data structures.

In a recent paper, Eliasi and Iranmanesh \cite{ei11} computed the hyper-Wiener index \cite{rt93b}, defined as $WW(G)=\frac{1}{2}W(G)+\frac{1}{2}M_1(G)$, of the `generalized hierarchical product' of graphs \cite{bdfm09}. (In fact, in \cite{bdfm09} the probabilistic method was used for computing the mean distance, and hence the Wiener index, of such a product.)

\begin{figure}[t]
\begin{center}
\vskip-1cm
\includegraphics[width=13cm]{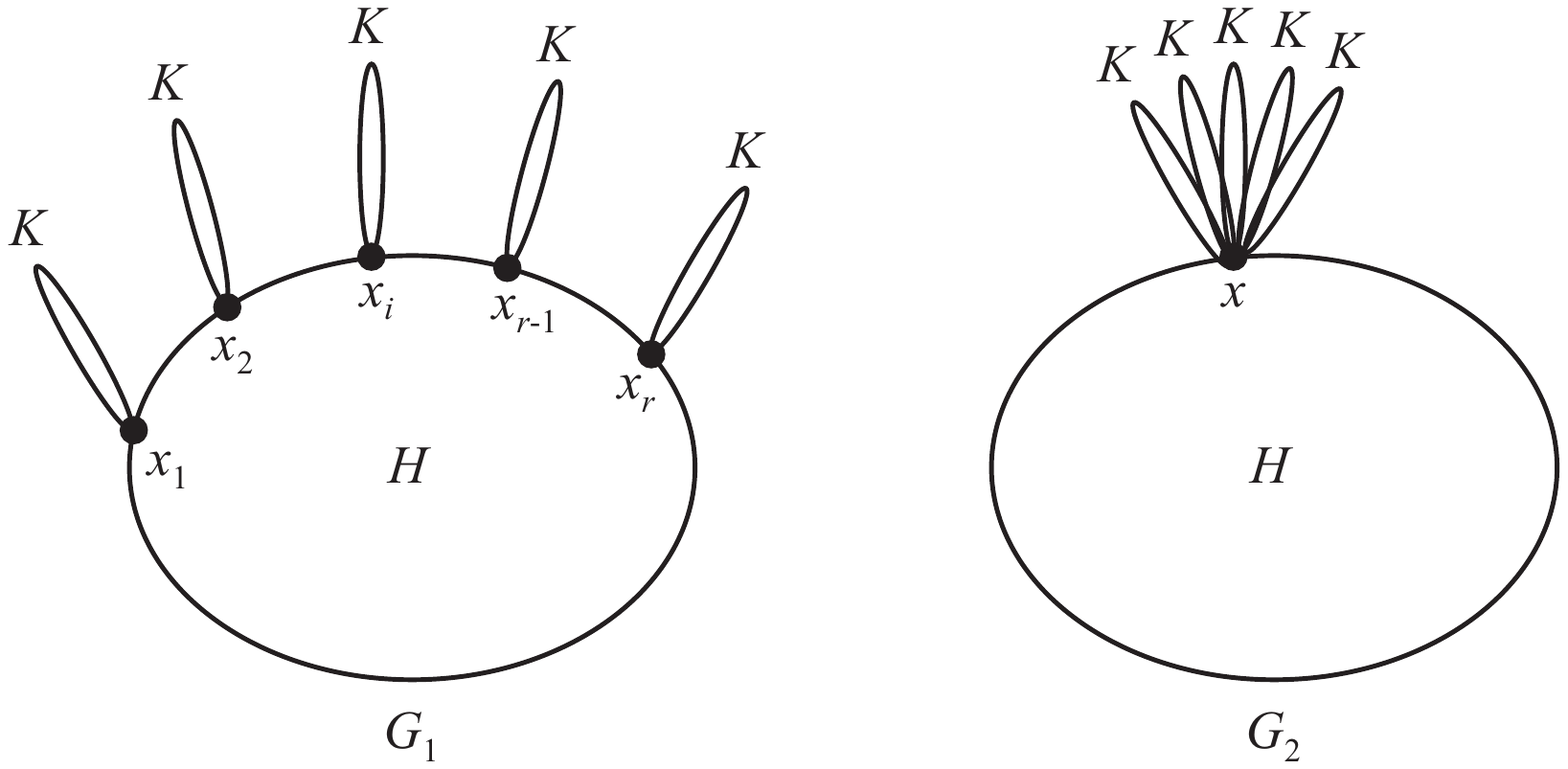}
\end{center}
\vskip -13cm
\caption{Two moment-related graft products}
\label{fig3-moments}
\end{figure}

\subsection{Comparison between moments}
Let $H$, $K$ be graphs with respective weight functions $\alpha$, $\beta$. Let $\gamma=\alpha+\beta+\stackrel{(r)}{\cdots}+\beta$. Consider $G_1$ defined as before and the particular case when $x_1=x_2=\cdots=x_r=x$, that is,
$$
G_2=H\left(
\begin{array}{ccc}
x & \cdots & x \\
y & \cdots & y
\end{array}
\right)(K,\ldots,K),
$$
see $G_2$ in Figure \ref{fig3-moments}.

\begin{figure}[t]
\begin{center}
\vskip-1cm
\includegraphics[width=16cm]{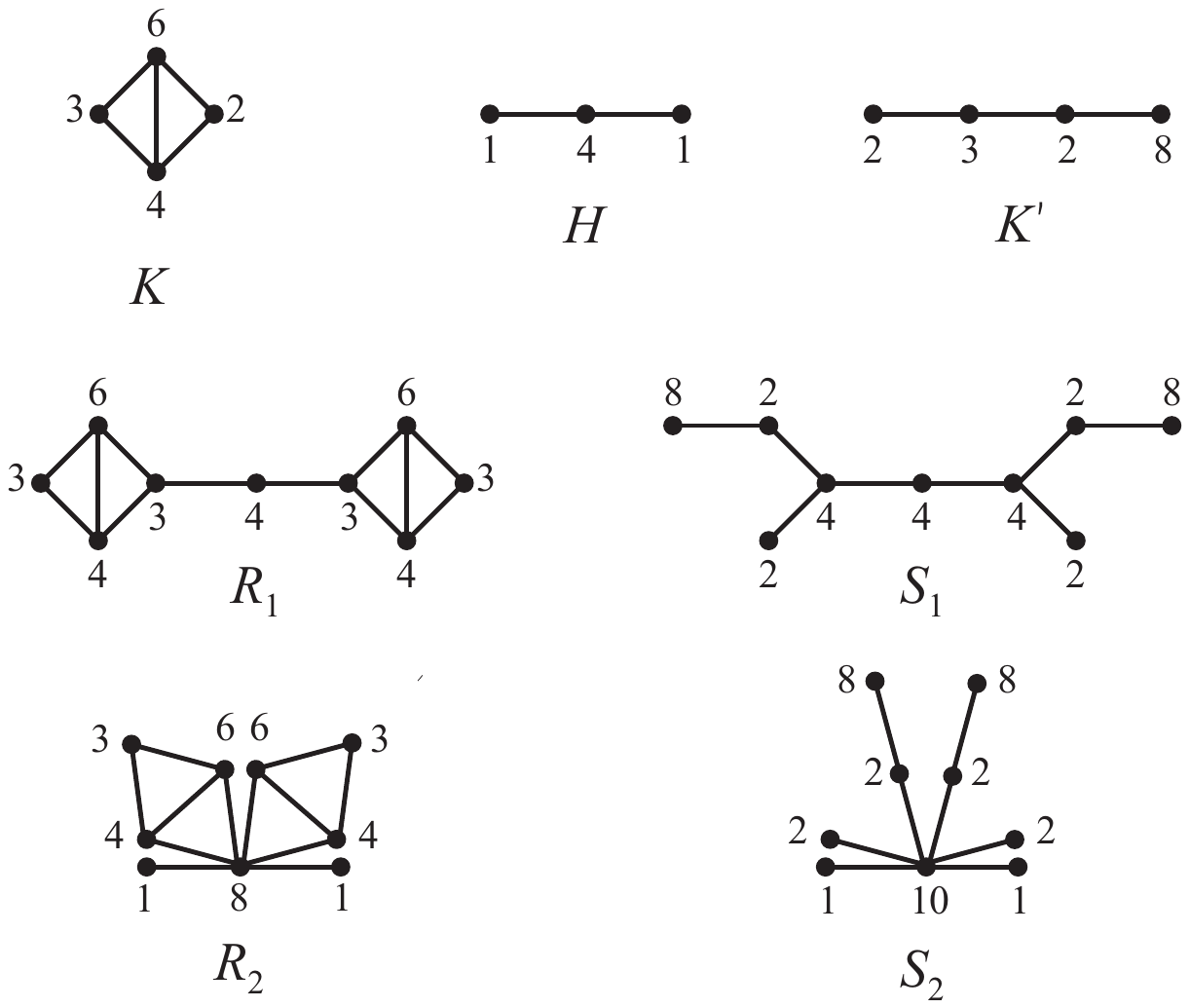}
\vskip -13.75cm
\caption{A comparison between moments}
\label{fig4-moments}
\end{center}
\end{figure}

Then we have the following result:
\begin{coro}
The difference of the moments of the graphs $G_1$ and $G_2$ defined above, both with weight function $\gamma=\alpha+\beta+\cdots+\beta$, satisfies
$$
M_{G_2}^{\gamma}-M_{G_1}^{\gamma}=\sum_{i=1}^r[M_H^{\xi}(x)-M_H^{\xi}(x_i)]-B(|V_K|-1)\sum_{i,j=1}^r \dist(x_i,x_j),
$$
where $\xi=(|V_K|-1)\alpha+B$.
\end{coro}
\begin{proof}
By Theorem \ref{basic-theo}, we get
\begin{equation}
\label{comparison1}
  M_{G_1}^{\gamma}=M_H^{\alpha}+rM_{K}^{\beta}+\sum_{i=1}^r M_H^{\xi}(x_i)+rM_K^{\eta}(y)+B(|V_K|-1)\sum_{i,j=1}^r \dist(x_i,x_j)
\end{equation}
with $\xi=(|V_K|-1)\alpha+B$ and $\eta=|V|+(r-1)B+(|V|+(r-1)|V_K|-r)\beta$.

The moment of $G_2$ is  obtained by considering the case $x_1=\cdots = x_r=x$:
\begin{equation}
\label{comparison2}
M_{G_2}^{\gamma}=M_H^{\alpha}+rM_{K}^{\beta}+r M_H^{\xi}(x)+rM_K^{\eta}(y).
\end{equation}
Then, the result follows from (\ref{comparison1}) and (\ref{comparison2}).
\end{proof}

Then, {\em the variation of the moment caused by $K$ only depends on its order and its total weight} (and neither on its weight function $\beta$ nor on its structure). By way of example, consider the graphs $K$ and $K'$, on four vertices and common weight $B=15$, depicted in Figure \ref{fig4-moments}. Their corresponding weight functions, as well as the graph $H$, have been arbitrarily chosen. We connect two copies of $K$ (respectively, $K'$) to $H$. First to the extreme vertices $x_1,x_2$, and then to the central vertex $x$. Then, $\sum_{i,j=1}^r \dist(x_i,x_j)=4$ and $\xi=3\alpha+15$.
Then, from the above comment, we see that moment differences coincide:
\begin{eqnarray*}
M_{R_2}^{\gamma}-M_{R_1}^{\gamma} & = & M_{S_2}^{\gamma'}-M_{S_1}^{\gamma'} \\
& = & 2M_H^{\xi}(x)-M_H^{\xi}(x_1)-M_H^{\xi}(x_2)-3\cdot 15\cdot 4= 72-126-180\\
&=&-234.
\end{eqnarray*}

\section{Trees and cycles}
In this section we consider a slight generalization of the graft product, together with its corresponding result for computing its moment, which leads to a more symmetric and compact presentation of Theorem \ref{basic-theo}. The proof is based on the fact that the reasoning given before allows the `receptor' vertices not to be necessarily different. Then, we only need  to translate the result to the new notation.

Let us consider a connected graph $H$ and a finite family $\F$ of disjoint connected graphs $K$. Fix one vertex $y_K\in V_K$ for each $K\in \F$ and consider a map $\F\rightarrow V_H$ defined by $K\mapsto x_K$. Let $\F_x$ be the anti-image of $x$ (that could be void). Then, the graft product $G=(V,E)$ is constructed by joining the graphs in $\F$ to $H$ by identifying each vertex $x\in V_H$ with the vertex $y_K$ of each $K\in\F_x$. This graph $K$, which shares vertex $x$ with $V_H$, has $n_x=\sum_{K\in\F_x} |V_k|-\F_x+1$ vertices. In particular, notice that if $\F_x=\emptyset$, then $n_x=1$. Also $|V|=\sum_{x\in V_H} n_x$. Let $\alpha$ and $\beta_K$, for each $K\in\F$, be some weight functions defined on the vertices of $H$ and $K$, respectively. The weight of $H$ and $K$ are denoted, respectively, by $A$ and $B_K$. The weight of the graphs attached to $x$ is then $w_x=\sum_{K\in\F_x} B_K$ and the total weight of $G$ is $W=A+\sum_{K\in\F} B_K$. On $G$ we consider the weight function
$\gamma=\alpha_H+\sum_{K\in\F}\beta_K$. Note that if $\alpha$ and $\beta_K$ are degree functions, then $\gamma$ also is.

Let  $\j,\n,\w$ be the (column) vectors with components $1$, $n_x$, and $w_x$, $x\in V_H$, respectively. Moreover, let $\D$ be the distance matrix with entries $(\D)_{xx'}=\dist(x,x')$ for every $x,x'\in V_H$.
Then, Theorem \ref{basic-theo} reads as follows:
\begin{theo}
\label{teorema4.1}
The moment of the graft product $G$ defined above is
\begin{equation}\label{basic-theo(b)}
 M_{G}^{\gamma}  =  M_H^{\alpha}+\sum_{K\in\F} M_{K}^{\beta_K}+\sum_{x\in V_H} M_{H}^{\xi_x}(x)+\sum_{K\in \F} M_{K}^{\eta_K}(y_K)
              +  (\n-\j)^{\top}\D\w,
\end{equation}
where $\xi_x=(n_x-1)\alpha+w_x$ and $\eta_K=(|V|-|V_K|)\beta_K+W-B_K$.
\end{theo}

\subsection{Unicyclic graphs}
Besides trees (see, for instance,
Dobrynin, Entringer, and  Gutman \cite{deg01}), unicyclic graphs have deserved a special attention in our context. For instance, A.I. Tomescu \cite{t08} gave lower bounds for the degree distance of (connected) unicyclic (and bicyclic) graphs, and characterized the extremal cases. (His result was generalized by Bucicovschi and Cioab\u{a} \cite{bc08} for connected graphs of given numbers of vertices and edges.)

Let $H=C_r$ be the cycle with vertices $x_1,x_2,\ldots,x_r$, and let $K_i=T_i$ be a tree on $n_i=|V_i|$ vertices, $i=1,\ldots,r$. With $y_i\in V_i$ for $i=1,\ldots,r$, consider the {\em unicyclic graph}
$$
G=C_r\left(
\begin{array}{ccc}
x_1 & \cdots & x_r \\
y_1 & \cdots & y_r
\end{array}
\right)(T_1,\ldots,T_r).
$$
To derive the moment of $G$, we need a simple lemma whose proof is immediate if we distinguish the cases of even and odd $r$.
\begin{lema}
\label{lema}
Let $C_r$ be the cycle with vertices $x_1,\ldots,x_r$. The distance matrix $\D_r$ with entries $(\D_r)_{ij}=\dist(x_i,x_j)$ has maximum eigenvalue $\theta_r=\lfloor\frac{r}{2}\rfloor\lfloor\frac{r+1}{2}\rfloor$ with (unique) associated eigenvector $\j$.
\end{lema}

Then, Theorem \ref{teorema4.1} yields the following result:

\begin{propo}
The moment with respect to the degree function $M_G^{\delta}$ (or degree distance $D'(G)$) of the unicyclic  graph, constructed by adding to the cycle $C_r$, according to the mapping $\F\rightarrow V_{C_r}$, the trees of a forest $\F$ thorough the vertices $\{y_K\}_{K\in\F}$, is
\begin{equation}
\label{moment-unicyclic}
M_{G}^{\delta}  =  \sum_{T\in \F} M_{T}^{\delta}+\sum_{T\in\F}^r M_{T}^{\eta_T}(y_T)
              +  2\n^{\top}\D\n,
\end{equation}
where $\eta_T=(|V|-|V_T|)(\delta+2)+2$.
\end{propo}
\begin{proof}
Let us compute (\ref{basic-theo(b)}) in our particular case:
\begin{equation}\label{(9)}
 M_{C_r}^{\delta}=\sum_{x,y\in V_{C_r}} \delta(x)\dist(x,y)=2r\sum_{y\in V_{C_r}} \dist(x,y)=2r\theta_r.
\end{equation}
For each tree, the total degree weight (sum of degrees) is twice its number of vertices minus two. Therefore,
$\xi_x=(n_x-1)\alpha + w_x = 2(n_x-1)+2(n_x-1)=4(n_x-1)$. Then,
\begin{equation}\label{(10)}
\sum_{x\in C_r} M_{C_r}^{\xi_x}(x)=4\sum_{x\in C_r}(n_x-1) \sum_{y\in C_r}\dist(x,y)
=4\theta_r\sum_{x\in C_r}(n_x-1)= 4\theta_r(|V|-r).
\end{equation}
From the degree weight of a tree and Lemma \ref{lema}, the last term in  (\ref{basic-theo(b)}) is:
\begin{eqnarray}
 2(\n-\j)^{\top}\D(\n-\j) &=& 2\n^{\top}\D\n-4 \n^{\top}\D\j+2\j^{\top}\D\j \nonumber\\
  & = &  2\n^{\top}\D\n -4 \theta_r\n^{\top} \j+2\theta_r\j^{\top}\j\nonumber \\
  & = & 2\n^{\top}\D\n-4\theta_r|V|+2r\theta_r. \label{(11)}
\end{eqnarray}
By adding up (\ref{(9)}), (\ref{(10)}) and (\ref{(11)}) we get the term $2\n^{\top}\D\n$ in (\ref{moment-unicyclic}).
Finally, since $W=2r+\sum_{T\in\F}2(|V_T|-1)=2r+2\sum_{x\in V_{C_r}} (n_x-1)=2r + 2(|V|-r)=2|V|$, we obtain
$$
\eta_T=(|V|-|V_T|)\delta+2|V|-2(|V_T|-1)=(|V|-|V_T|)(\delta+2)+2,
$$
which completes the proof.
\end{proof}

\subsection{Extended cycles}
We call {\rm extended cycles} the family of ordinary cycles
($r> 2$), edges ($r=2$) and singletons ($r=1$). Thus, an extended cycle of $r$ vertices has $r$, $1$ or $0$ edges, and degree $2$, $1$ or $0$, respectively, depending on the case and, by Lemma \ref{lema}, its distance matrix has maximum eigenvalue $\theta_r=\lfloor\frac{r}{2}\rfloor\lfloor\frac{r+1}{2}\rfloor$ for $r\ge 2$ and $\theta_1=0$.

When we consider the graft product of extended cycles, we obtain the following result:

\begin{propo}
Let $C$ be an extended cycle on $r=|V_C|$ vertices. For each vertex $x\in V_C$, consider an extended cycle $C_x$ with $r_x=|V_x|$ vertices and $m_x$ edges. Let $G=(V,E)$ be the graft product obtained by amalgamating each vertex $x$ with a vertex $y_x\in V_x$.  Then, the moment of $G$ with respect to the degree function (or degree distance) is:
\begin{equation}
\label{M-cycles}
  M_G^{\delta}=2\left(m[\vectheta]+[\m][\vectheta]-\langle \m,\vectheta \rangle\right)+(2\theta+\langle \vecdelta,\vectheta\rangle)[\r]+2\r^{\top}\D\m,
\end{equation}
where $m=|E|$, $\D$ is the distance matrix of $C$, $\r,\m,\vectheta,\vecdelta$ are, respectively, the vectors with components $r_x,m_x,\theta_x(=\theta_{r_x}),\delta_x$ for $x\in V_C$, and $[\v]=\langle \v,\j\rangle$.
\end{propo}

\begin{proof}
Let us compute the different terms of the expression of $M_G^{\delta}$ given by Theorem \ref{basic-theo}:
\begin{eqnarray}
M_C^{\delta} &\!\!\!=\!\!\!& \sum_{x\in V_C}\delta\sum_{y\in V_C}\dist(x,y)=\sum_{x\in V_C}\delta\theta=r\delta\theta. \label{(12)} \\
M_{C_x}^{\delta}  &\!\!\!=\!\!\!&  \sum_{y\in V_x}\delta_x\sum_{z\in V_x}\dist(y,z)=\sum_{y\in V_x}\delta_x\theta_x=r_x\delta_x\theta_x=2m_x\theta_x, \nonumber \\
\sum_{x\in V_C} M_{C_x}^{\delta} &\!\!\!=\!\!\!& 2 \sum_{x\in V_C}m_x\theta_x=2\langle \m,\vectheta\rangle. \label{(13)}\\
\xi_x &\!\!\!=\!\!\!& (|V_x|-1)\delta+B_x=(r_x-1)\delta+2m_x, \nonumber\\
M_C^{\xi_x} &\!\!\!=\!\!\!& \sum_{y\in V_C} [(r_x-1)\delta+2m_x]\dist(x,y)=[(r_x-1)\delta+2m_x]\theta,\nonumber\\
\sum_{x\in V_C} M_C^{\xi_x} &\!\!\!=\!\!\!& \delta\theta\sum_{x\in V_C}r_x-r\delta\theta + 2\theta \sum_{x\in V_C} m_x=\delta\theta[\r]-r\delta \theta+2\theta[\m]. \label{(14)}\\
\eta_x &\!\!\!=\!\!\!& (|V|-|V_x|)\beta_x + W -B_x \nonumber\\
 &\!\!\!=\!\!\!& \left(\sum_{y\in V_C}r_y-r_x\right)\delta_x+2m+2\sum_{y\in V_C}m_y-2m_x\nonumber \\
&\!\!\!=\!\!\!& \delta_x[\r]-4m_x+2m+2[\m],\nonumber \\
 M_{C_x}^{\eta_x}(y_k) &\!\!\!=\!\!\!&  \delta_x\theta_x[\r]+2m\theta_x-4m_x\theta_x+2[\m]\theta_x, \nonumber \\
\sum_{x\in V_C} M_{C_x}^{\eta_x}(y_k) &\!\!\!=\!\!\!& \langle\vecdelta,\vectheta_x\rangle[\r]+2m[\vectheta]-4\langle \m,\vectheta\rangle+2[\m][\vectheta]. \label{(15)}\\
2(\r-\j)^{\top}\D\m &\!\!\!=\!\!\!& 2\r^{\top}\D\m-2\theta[\m] \label{(16)}.
\end{eqnarray}
Then, the result follows by adding expressions from
$(\ref{(12)})$ to $(\ref{(16)})$.
\end{proof}

In the case when $C$ and $C_x$ are proper cycles $(r\ge 3, r_x\ge 3)$ for all $x\in V_C$, we get the following:

\begin{propo}
The degree distance $D'(G)=M_G^{\delta}$ of the graft product of cycles $C$ and $C_x$, $x\in V_C$, is
$$
M_G^{\delta}=4\left(\sum_{x\in V_C}r_x\right)\left(\sum_{x\in V_C}\theta_x\right)+2\left(\theta\sum_{x\in V_C}r_x+r\sum_{x\in V_C}\theta_x-\sum_{x\in V_C}r_x\theta_x\right)+2\r^{\top}\D\r.
$$
\end{propo}
\begin{proof}
Just observe that, under the hypothesis, $m=r$, $\m=\r$ and $\vecdelta=2\j$.
\end{proof}

\vskip 1cm
\noindent{\large \bf Acknowledgments.}  Research supported by the {\em Ministerio de Ciencia e Innovaci\'on}, Spain, and the {\em European Regional Development Fund} under project MTM2011-28800-C02-01 and by the {\em Catalan Research Council} under project 2009SGR1387.
We thank Prof. J.L.A. Yebra for his useful comments on this paper.


\end{document}